\documentclass[12pt,letterpaper]{article}

\usepackage{amssymb,amsfonts,amscd,amsthm}
\usepackage[all,arc]{xy}
\usepackage{enumerate, bbm}
\usepackage{mathrsfs}
\usepackage{tikz}
\usepackage[left=0.9in,top=0.9in,right=0.9in,bottom=0.9in]{geometry}
\usepackage{mathtools}
\usepackage{hyperref}
\usepackage{color}
\usepackage{graphicx}
\usepackage{enumitem} 
\graphicspath{ {TexImages/} }

\newtheorem{thm}{Theorem}[section]

\newtheorem{prop}[thm]{Proposition}
\newtheorem{lem}[thm]{Lemma}

\newtheorem{exmp}[thm]{Example}

\newtheorem{conj}[thm]{Conjecture}

\newtheorem{ques}[thm]{Question}

\theoremstyle{definition}
\newtheorem{defn}{Definition}

\hypersetup{
	colorlinks,
	citecolor=blue,
	filecolor=blue,
	linkcolor=blue,
	urlcolor=blue,
	linktocpage
}

\setenumerate[1]{label=\thesection.\arabic*.} 
\setenumerate[2]{label*=\arabic*.} 
\setlist[enumerate]{itemsep=2ex, topsep=2ex} 
\setlist[itemize]{itemsep=2ex, topsep=2ex}

\renewcommand{\l}{\left}
\renewcommand{\r}{\right}

\newcommand{\floor}[1]{\l\lfloor #1\r\rfloor}

\renewcommand{\l}{\mathcal{P}_2}

\renewcommand{\l}{\left}

\title{New Bounds on Diffsequences}
\author{ Alexander Clifton\footnote{Discrete Mathematics Group, Institute for Basic Science, Daejeon, South Korea {\tt yoa@ibs.re.kr}. The author is supported by the Institute for Basic Science (IBS-R029-C1) and in part by NSF Grant DMS-1945200.}}
\date{\today}
\begin{document}
	\maketitle
	\begin{abstract}
	 For a set of positive integers $D$, a $k$-term $D$-diffsequence is a sequence of positive integers $a_1<a_2<\dots<a_k$ such that $a_i-a_{i-1}\in D$ for $i=2,3,\dots,k$. For $k\in\mathbb{Z}^+$ and $D\subset \mathbb{Z}^+$, we define $\Delta(D,k)$, if it exists, to be the smallest integer $n$ such that every $2$-coloring of $\{1,2,\dots,n\}$ contains a monochromatic $D$-diffsequence of length $k$. We improve the lower bound on $\Delta(D,k)$ where $D=\{2^i\mid i\in\mathbb{Z}_{\geq{0}}\}$, proving a conjecture of Chokshi, Clifton, Landman, and Sawin. We also determine all sets of the form $D=\{d_1,d_2,\dots\}$ with $d_i\mid d_{i+1}$ for which $\Delta(D,k)$ exists.
	\end{abstract}
	\section{Introduction}\label{sec1}

	Ramsey theory questions typically ask about whether a graph contains a monochromatic copy of a particular subgraph when the edges are colored with $r$ colors. However, more generally, the question at the heart of Ramsey Theory is whether or not a set retains a certain property when the elements are divided into disjoint subsets. For example, the positive integers contain arbitrarily long arithmetic progressions, but if we split them into $r$ disjoint sets, will there still be a set with arbitrarily long arithmetic progressions? In 1927, van der Waerden \cite{Wae27} answered this question in the affirmative by showing that every $r$-coloring of the natural numbers will contain arbitrarily long monochromatic arithmetic progressions. 

It is natural to ask what other types of structures besides arithmetic progressions exhibit van der Waerden-type theorems. Van der Waerden's work is predated by Schur \cite{Sch17} who showed that every $r$-coloring of $\mathbb{Z}^+$ contains some monochromatic triple $(x,y,z)$ satisfying $x+y=z$. More recent sequences of study include arithmetic progressions modulo $m$ \cite{LW97,LL94}, arithmetic progressions with a fixed common difference \cite{LW97,LL94, BGL99}, and  descending waves \cite{BEF90}.

Landman and Robertson considered the question of whether an $r$-coloring must produce a monochromatic sequence whose gaps lie in some fixed set. The answer to this question will depend on the choice of $r$ as well as the choice of the set of gaps. In \cite{LR03}, Landman and Robertson introduced the notion of a $D$- diffsequence:

	
	\begin{defn}
	For a set $D$ of positive integers, a \textit{$D$-diffsequence} of length $k$ is a sequence of positive integers $a_1<a_2\dots<a_k$ such that each gap lies in $D$. That is,
	\[a_i-a_{i-1}\in D,\]
	for $i=2,3,\dots, k$.
	\end{defn}
	
	\begin{defn}
	A set $D$ is called \textit{$r$-accessible} if every $r$-coloring of the positive integers contains arbitrarily long monochromatic $D$-diffsequences.
	\end{defn}
	
	Note that any $r$-accessible set $D$ is automatically $t$-accessible for $t<r$. The notion of $r$-accessibility is analogous to van der Waerden's theorem with $r$ colors.
	
	The set of Fibonacci numbers is known to be $2$-accessible \cite{LR03} but not $4$-accessible \cite{Wes22}. The set of primes is not $3$-accessible \cite{LR03} but the question of whether or not it is $2$-accessible remains unresolved. Results on accessibility for fixed translates of the set of primes are given in \cite{LR03} and \cite{LV10}. The set of powers of $n$ is only $2$-accessible when $n=2$ since otherwise there is a periodic coloring modulo $n-1$ that avoids arbitrarily long $D$-diffsequences.
	
	Just as the van der Waerden number $W(r,k)$ is the smallest $n$ such that any $r$-coloring of $[n]$ contains a monochromatic arithmetic progression of length $k$, we can consider analogues in the diffsequence setting.
	If $D$ is $r$-accessible, we let $\Delta(D,k;r)$ denote the smallest positive integer $n$ such that any $r$-coloring of $[n]:=\{1,2,\dots,n\}$ contains a monochromatic $D$-diffsequence of length $k$. For simplicity, we let $\Delta(D,k):=\Delta(D,k;2)$.
	
	Landman and Robertson showed that the set $D=\{2^i\mid i\in\mathbb{Z}_{\geq{0}}\}$ is $2$-accessible with $\Delta(D,k)\leq{2^k-1}$. Chokshi, Clifton, Landman, and Sawin \cite{CCLS18} conjectured the existence of an exponential lower bound. We verify this conjecture, establishing that $\Delta(D,k)$ is at least exponential in $\sqrt{k}$.
	\begin{thm}\label{main}
	When $D=\{2^i\mid i\in\mathbb{Z}_{\geq{0}}\}$,
	\[
	\Delta(D,k)\geq{2^{\sqrt{2k}}\left(\frac{(\sqrt{2}-1)k}{8}-\frac{\sqrt{k}}{8}\right)+\frac{\sqrt{k}}{2}}.
	\]
	\end{thm}
	
	We next consider some sets that are not $2$-accessible.
	
	\begin{prop}\label{fact}
	The set $D=\{k!\mid k\in\mathbb{Z}^+\}$ is not $2$-accessible.
	\end{prop}
	Generalizing both powers of $2$ and factorials, we consider sets of the form \[D_{\{a_n\}}:=\bigg\{\displaystyle\prod_{i=1}^k a_i\mid k\in\mathbb{Z}^+\bigg\},\] for a fixed sequence of positive integers $a_1,a_2,a_3,\dots$ with $a_i\geq{2}$ whenever $i\geq{2}$. The elements of the set $D_{\{a_n\}}$, arranged in increasing order, form a \textit{dividing sequence} where each term is a factor of all subsequent terms.
	
	Note that $\{a_n\}=(1,2,2,2,2,\dots)$ recovers powers of $2$ while $\{a_n\}=(1,2,3,4,5,\dots)$ recovers factorials.
	
	We are able to completely characterize when $D_{\{a_n\}}$ is $2$-accessible:
	\begin{thm}\label{1.3}
	$D_{\{a_n\}}$ is $2$-accessible if and only if $\{a_n\}$ contains arbitrarily long strings of consecutive $2$'s.
	\end{thm}
	
	We will prove Theorem \ref{main} in Section \ref{sec1}, Proposition \ref{fact} in Section \ref{prev}, and Theorem \ref{1.3} in Section \ref{sec4}. In Section \ref{cr}, we propose some further questions.

	\section{Powers of $2$}\label{sec2}
	Throughout this section, we let $D=\{2^i\mid i\in\mathbb{Z}_{\geq{0}}\}$. We will prove Theorem \ref{main} by utilizing a series of periodic colorings to obtain the lower bound. We begin by defining a family of colorings, $P_t$, for $t\in\mathbb{Z}^+$. These colorings will not quite yield the desired bound but serve as the basis for another family of colorings which will.
	
	Let $P_1$ be a periodic coloring with a repeating block of ``$10$". This means that odd numbers are assigned color $1$ and even numbers are assigned color $0$.
	
	Let $P_2$ be a periodic coloring with a repeating block of ``$1001$". This means numbers that are $0$ or $1\pmod{4}$ are color $1$ while numbers that are $2$ or $3\pmod{4}$ are color $0$.
	
	Continuing in this manner, we define $P_t$ as a periodic coloring with a repeating block of size $2^t$. The first half of the repeating block is the repeating block of $P_{t-1}$ while the second half is the complement of the first half. Thus, by construction, any two numbers $x$ and $x+2^{t-1}$ have different colors with respect to $P_t$.
	
	Note that the repeating block of $P_t$ can be thought of as the first $2^t$ bits of the Thue-Morse sequence and that $P_3$ is the periodic coloring that Landman and Robertson used to obtain their initial linear lower bound in \cite{LR03}.
	
	A monochromatic $D$-diffsequence $a_1<a_2<\dots<a_k$ with respect to the coloring $P_t$ has no gaps of size $2^{t-1}$. Furthermore, we can bound from above the number of gaps of each size less than $2^{t-1}$ in a way that does not depend on $k$.

\begin{lem}\label{smallgaps}
For $t\geq{2}$, the following holds for any monochromatic $D$-diffsequence under the coloring $P_t$: For $m=0,1,\dots,t-2$, there are at most $m+1$ gaps of size $2^m$.
\end{lem}
\begin{proof}
Consider a monochromatic $D$-diffsequence with respect to $P_t$. For each $m=0,1,\dots,{t-2}$, we can split each repeating block of size $2^t$ into $2^{t-m-1}$ disjoint sub-blocks of size $2^{m+1}$. Note that a sub-block of size $2^{m+1}$ looks like a repeating block of $P_{m+1}$ (possibly with colors swapped). Thus, by construction, if $a_{i+1}-a_i=2^m$, then $a_i$ cannot be in the first half of any sub-block of size $2^{m+1}$. Therefore, for any gap $a_{i+1}-a_i$ of size $2^m$, $a_i\in\{ 2^m+1,2^m+2,\dots,2^{m+1}\}\pmod{2^{m+1}}$ and $a_{i+1}\in\{1,2,\dots,2^m\}\pmod{2^{m+1}}$. We say that $a_i$ is in the \textit{second half} mod $2^{m+1}$ while $a_{i+1}$ is in the \textit{first half}.
 
 Keeping with the assumption that $a_{i+1}-a_i=2^m$, if the gaps $a_{i+2}-a_{i+1},a_{i+3}-a_{i+2},\dots, a_{j}-a_{j-1}$ are all larger than $2^m$, then $a_j$ is in the first half mod $2^{m+1}$ and thus $a_{j+1}-a_j$ cannot be a gap of size $2^m$. This means that in between every two gaps of the same size, there is at least one gap of a smaller size. Since there is no way to have a gap of size smaller than $1$, there is at most one gap of size $1$, establishing Lemma \ref{smallgaps} for $m=0$.
 
 We will now finish the proof via induction. Suppose that there exists some $1\leq{q}\leq{t-2}$ such that there are at most $m+1$ gaps of size $2^m$ for $m=0,1,2,\dots,q-1$. We will now show that there are at most $q+1$ gaps of size $2^q$.
 
 Assume for the sake of contradiction that there are at least $q+2$ gaps of size $2^q$. The first $q+2$ of these are the gaps between $a_{i_j}$ and $a_{i_j+1}$ for some indices $i_j$ with $j=1,2,\dots,q+2$. For $j=1,2,\dots,q+1$, the sequence of gaps $a_{i_j+2}-a_{i_j+1},a_{i_j+3}-a_{i_j+2},\dots,a_{i_{j+1}}-a_{i_{j+1}-1}$, consists of at least one power of $2$ smaller than $2^q$, as well as possibly some powers of $2$ larger than $2^q$, which we can ignore, as they are all $0\pmod{2^{q+1}}$.
 
 For $j=1,2,\dots,q+1$, let $A_j$ denote the sum of those gaps among $a_{i_j+2}-a_{i_j+1},{a_{i_j+3}-a_{i_j+2}},\dots,\\
 {a_{i_{j+1}}-a_{i_{j+1}-1}}$ that are less than $2^q$. Consider $a_{i_1}$, which we know to be in the second half mod $2^{q+1}$. Similarly, every $a_{i_j}$ is in the second half mod $2^{q+1}$, for $j=2,3,\dots,q+2$. Hence, for each $j=1,2,\dots,q+1$, we have that $(a_{i_1}+2^q)+(A_1+2^q)+(A_2+2^q)+\dots+(A_{j-1}+2^q)+A_j=a_{i_1}+(A_1+A_2+\dots+A_j)+j2^q$ is in the second half mod $2^{q+1}$. This means that $a_{i_1}+(A_1+A_2+\dots+A_j)$ is in the second half mod $2^{q+1}$ if $j$ is even and is in the first half mod $2^{q+1}$ if $j$ is odd.
 
 Beginning with $a_{i_1}$, which is in the second half mod $2^{q+1}$, every time we add the next $A_j$, we switch halves mod $2^{q+1}$. This means that $A_1+A_2+\dots+A_r>(r-1)2^q$ for any $r=1,2,\dots,q+1$.
 In particular, $A_1+A_2+\dots+A_{q+1}>q2^q$. So there exist gaps of sizes $1,2,4,\dots,2^{q-1}$ that collectively sum to at least $q2^q+1$. By the inductive hypothesis, we have at most $m+1$ gaps of size $2^m$ for $m=0,1,\dots,q-1$.
 Thus the gaps of size $1,2,4,\dots,2^{q-1}$ collectively sum to at most:
 \begin{align*}
     \displaystyle\sum_{m=0}^{q-1} (m+1)2^m&=\sum_{m=0}^{q-1}(2^q-2^m)\\
     &=q2^q-\sum_{m=0}^{q-1}2^m\\
     &=q2^q-(2^q-1)\\
     &=(q-1)2^q+1.
 \end{align*}
 This contradicts that these must sum to at least $q2^q+1$, so it is impossible for a monochromatic $D$-diffsequence with respect to $P_t$ to have $q+2$ gaps of size $2^q$. Thus, there are at most $m+1$ gaps of size $2^m$ for $m=0,1,2,\dots,q$ and the induction is complete.
\end{proof}

	Lemma \ref{smallgaps} gives a constant upper bound, $1+2+\dots+(t-1)=\frac{t(t-1)}{2}$, on the number of gaps smaller than $2^t$ in a monochromatic $D$-diffsequence with respect to $P_t$ and thus a lower bound of $k-1-\frac{t(t-1)}{2}$ on the number of gaps of size at least $2^t$. Summing the sizes of these gaps gives a lower bound for $\Delta(D,k)$ that is linear in terms of $k$ and whose slope depends on $t$:
	\[
	\Delta(D,k)\geq{2^t(k-1-\frac{t(t-1)}{2})}.
	\]
	Picking $t=\lfloor{\sqrt{2k}\rfloor}$ for a given $k$ is in fact enough to obtain a lower bound of \[\Delta(D,k)\geq{{2^{\lfloor\sqrt{2k}\rfloor}(\frac{\sqrt{2k}}{2}-1)}}.\] This bound has the same exponent as in Theorem \ref{main}, but differs by a factor of $\Theta(\sqrt{k})$. Building on our family of $P_t$ colorings, we will demonstrate a more refined series of colorings that achieves a slightly better lower bound.
	
	For $u\in\mathbb{Z}_{\geq{0}}$, let $P_{t,u}$ be the periodic coloring with a repeating block of size $2^{t+u}$ obtained by replacing each bit of $P_t$ by $2^u$ copies of itself. We call these $2^u$ consecutive $1$'s or $2^u$ consecutive $0$'s a \textit{sub-block}.  Note that $P_{t,0}$ is simply ${P_t}$.
	
	\begin{exmp}
	    $P_{3,1}$ has a repeating block of ``$1100001100111100$", containing $8$ sub-blocks of length $2$.
	 \end{exmp}
\begin{exmp}$P_{2,2}$ has a repeating block of ``$1111000000001111$", containing $4$ sub-blocks of length $4$.
	\end{exmp}
	
	We now use the family of $P_{t,u}$ colorings to prove Theorem \ref{main}.
	\begin{proof}
	If $a_i$ and $a_{i+1}$ are consecutive entries in a monochromatic $D$-diffsequence with respect to the coloring $P_{t,u}$, there are three possibilities: 
\begin{itemize}
    \item[i)] $a_i$ and $a_{i+1}$ are in different positions within the same sub-block, or
    \item[ii)] $a_i$ and $a_{i+1}$ are in different sub-blocks (possibly in the same block) but the same position within their respective sub-blocks, or
    \item[iii)] $a_i$ and $a_{i+1}$ are in different sub-blocks and different positions within their respective sub-blocks.
\end{itemize}
	
	In the second case, $a_{i+1}-a_i$ is a multiple of $2^u$, the length of a sub-block, so $a_{i+1}$ appears $2^l$ sub-blocks after $a_i$ for some $l$. In the third case, we note that $a_{i+1}-a_i$ cannot be divisible by the sub-block size $2^u$. Therefore, it is at most $2^{u-1}$, so $a_i$ and $a_{i+1}$ are in consecutive sub-blocks.
	
	If we consider what sub-blocks $a_1,a_2,\dots,a_k$ lie in and ignore those that are not the first $a_i$ in their respective sub-block, this corresponds to a monochromatic $D$-diffsequence with respect to the coloring $P_t$.
	
	\begin{exmp}
	$5,6,10,11,12,28$ is a monochromatic $D$-diffsequence in color $0$ with respect to the coloring $P_{2,2}$. If we look at the sub-blocks these are in, we have $2,2,3,3,3,7$. Ignoring repeated occurrences of the same sub-block leaves us with $2,3,7$, which is a monochromatic $D$-diffsequence of color $0$ with respect to $P_2$.
	\end{exmp}
	
	Applying Lemma \ref{smallgaps} to this $D$-diffsequence gives that for $m=0,1,\dots,t-2$, there are at most $m+1$ occurrences of $a_i$ and $a_{i+1}$ in the original sequence at a distance of $2^m$ sub-blocks apart. In particular, there is at most one $i$ for which $a_i$ and $a_{i+1}$ are in consecutive sub-blocks. For $m=1,\dots,t-2$, there at most $m+1$ values of $i$ with $a_{i+1}-a_i=2^{m+u}$.
	
	In the first case, $a_{i+1}$ is in a later position within the sub-block than $a_i$ is, while in the second case, $a_{i+1}$ is in the same position within its respective sub-block as $a_i$. In the third case, $a_{i+1}$ is earlier in its sub-block than $a_i$ is, but since these are consecutive sub-blocks, this can happen at most once. Thus, if we keep track of changes in the position of $a_i$ within its sub-block, it can increase within the range $1$ to $2^u$, decrease once, then increase again, potentially up to $2^u$. This means the gaps of size less than $2^u$, accounted for by the first and third cases, sum to at most $2(2^u)-1$. For $m=0,\dots,t-2$, we have an upper bound on the number of gaps of size $2^{m+u}$. We can then obtain a lower bound on the number of gaps of size at least $2^{t+u}$. (Note that gaps of size $2^{t-1+u}$ cannot occur in a monochromatic $D$-diffsequence with respect to $P_{t,u}$.)
	
	Although it is possible for gaps of size $2^u$ to occur in a monochromatic $D$-diffsequence with respect to $P_{t,u}$, there is never any reason to use them if the goal is to minimize $a_k$. This is because a gap of size $2^u$ can only occur if two sub-blocks of all $1$'s or two sub-blocks of all $0$'s are adjacent. In that case, it is better to use $2^u$ consecutive gaps of size $1$ in place of one gap of size $2^u$. This means there are at most $\displaystyle\sum_{m=1}^{t-2}(m+1)$ gaps of size less than $2^{t+u}$ accounted for by the second case, in addition to at most $2^{u+1}-1$ gaps of size less than $2^{t+u}$ accounted for by the first and third cases. There are a total of $k-1$ gaps $a_{i+1}-a_i$, so at least
	
	\[
	(k-1)-\displaystyle\sum_{m=1}^{t-2}(m+1)-(2^{u+1}-1)
	\]
	
	have size at least $2^{t+u}$. Replacing a smaller gap by another gap of size at least $2^{t+u}$ would increase the sum of the gap sizes, so to minimize the total size of the gaps, we will take as many as possible of each size less than $2^{t+u}$.
	
	Thus, the sum of the $k-1$ gaps is at least:
	
	\begin{align*}
    &(2^{u+1}-1)+\displaystyle\sum_{m=1}^{t-2} (m+1)2^{m+u}+[k-1-(2^{u+1}-1)-\displaystyle\sum_{m=1}^{t-2} (m+1)]2^{t+u}\\=&2^{t+u}(k-2^{u+1}-\frac{(t-2)(t+1)}{2})+(2^{u+1}-1)+(2^{u+t-1}-2^{u+1})+\displaystyle\sum_{m=1}^{t-2}(2^{u+t-1}-2^{u+m})\\
    =&2^{t+u}(k-2^{u+1}-\frac{t^2}{2}+\frac{t}{2}+1)+(t-2)2^{u+t-1}+2^{u+1}-1\\
    =&2^{t+u}(k-2^{u+1}-\frac{t^2}{2}+t)+2^{u+1}-1.
\end{align*}

This sum is a lower bound for $a_k-a_1$ so 
\[
\Delta(D,k)\geq{2^{t+u}(k-2^{u+1}-\frac{t^2}{2}+t)+2^{u+1}}.
\]

We can now strategically choose $t$ and $u$ to optimize the lower bound. Roughly, we would like $t+u$ to be as large as possible, while keeping $k-2^{u+1}-\frac{t^2}{2}+t$ positive. Typically, increasing $u$ by one has more impact on $k-2^{u+1}-\frac{t^2}{2}+t$ than increasing $t$ by one does, so we prioritize maximizing $t$, which gives $\frac{t^2}{2}\approx k$. Then, we maximize $u$ by choosing $2^{u+1}\approx t$. Selecting $t=\floor{\sqrt{2k}}$ and $u=\floor{\log_2{\frac{\sqrt{k}}{2}}}$, we have:

\begin{align*}
    \Delta(D,k)&\geq{2^{\floor{\sqrt{2k}}+\floor{\log_2{\frac{\sqrt{k}}{2}}}}\left(k-2^{\floor{\log_2{\frac{\sqrt{k}}{2}}}+1}-\frac{\floor{\sqrt{2k}}^2}{2}+\floor{\sqrt{2k}}\right)+2^{\floor{\log_2{\frac{\sqrt{k}}{2}}}+1}}\\
    &\geq{\frac{2^{\sqrt{2k}}(\frac{\sqrt{k}}{2})}{4}\left(k-\sqrt{k}-k+\sqrt{2k}-1\right)+\frac{\sqrt{k}}{2}}=2^{\sqrt{2k}}\left(\frac{(\sqrt{2}-1)k}{8}-\frac{\sqrt{k}}{8}\right)+\frac{\sqrt{k}}{2}.
\end{align*}
	\end{proof}

	\section{Factorials}\label{prev}
	In this section, we let $D=\{k! \mid k\in \mathbb{Z}^+\}$.
We would like to show that the set, $D$, of factorials is not $2$-accessible. This means we must demonstrate a coloring of $\mathbb{Z}^+$ that avoids arbitrarily long $D$-diffsequences. We construct the following $2$-coloring, which can be thought of as a slight perturbation of a periodic coloring:
	
	The integer $n$ is assigned color $0$ or $1$ depending on the parity of $\lfloor{n\alpha\rfloor}$ where \[\alpha:=2-\frac{e}{2}-\frac{1}{2e}=1-\displaystyle\sum_{i=1}^\infty \frac{1}{(2i)!}.\]
	Note that the numbers colored $0$ are precisely the indices of the even terms in the Beatty sequence $\mathcal{B}_{2+\alpha}:=\floor{1(2+\alpha)}, \floor{2(2+\alpha)},\floor{3(2+\alpha)},\dots$. 
	
	Here we prove Proposition \ref{fact} by showing that this coloring avoids $D$-diffsequences of length $4$. That is, there are no $a_1<a_2<a_3<a_4$ in $\mathbb{Z}^+$ such that $\lfloor{a_i\alpha\rfloor}$ has the same parity for $i=1,2,3,4$ and that $a_{i+1}-a_i\in\{k!\mid k\in\mathbb{Z}^+\}$ for $i=1,2,3$.
	
\begin{lem}\label{1/3bound}
For each $k\in\mathbb{Z}^+$, there exists some integer $n$ such that $2n+\frac{1}{3}\leq{k!\alpha}<2n+1$.
\end{lem}
\begin{proof}
We split into the cases where $k$ is even and when $k$ is odd:

When $k=2m$ for $m\geq{1}$, we have
\begin{align*}
    k!\alpha&=(2m)!-\displaystyle\sum_{i=1}^{\infty}\frac{(2m)!}{(2i)!}
    =(2m)!-\displaystyle\sum_{i=1}^{m}\frac{(2m)!}{(2i)!}-\displaystyle\sum_{i=m+1}^{\infty}\frac{(2m)!}{(2i)!}\\
    &=(2m)!-\displaystyle\sum_{i=1}^{m-1}[(2m)(2m-1)\dots(2i+1)]-1-\displaystyle\sum_{i=m+1}^{\infty}\frac{1}{(2m+1)(2m+2)\dots(2i)}\\
    &=1-\displaystyle\sum_{i=m+1}^{\infty}\frac{1}{(2m+1)(2m+2)\dots(2i)}\pmod{2}.
\end{align*}

To reach the desired conclusion, it suffices to demonstrate that \[\frac{2}{3}\geq{\displaystyle\sum_{i=m+1}^{\infty}\frac{1}{(2m+1)(2m+2)\dots(2i)}}>0.\] It is clear that the sum is positive. It can also be bounded above by the following convergent geometric series:
\begin{align*}
    \displaystyle\sum_{i=m+1}^{\infty}\frac{1}{(2m+1)^{2i-2m}}&=\frac{\frac{1}{(2m+1)^2}}{1-\frac{1}{(2m+1)^2}}=\frac{1}{(2m+1)^2-1}\\
    &\leq{\frac{1}{3^2-1}}=\frac{1}{8}.
\end{align*}

When $k=2m+1$ for $m\geq{1}$, we have
\begin{align*}
k!\alpha&=(2m+1)!-\displaystyle\sum_{i=1}^{\infty}\frac{(2m+1)!}{(2i)!}=(2m+1)!-\displaystyle\sum_{i=1}^{m}\frac{(2m+1)!}{(2i)!}-\displaystyle\sum_{i=m+1}^{\infty}\frac{(2m+1)!}{(2i)!}\\
&=(2m+1)!-\displaystyle\sum_{i=1}^{m-1}\frac{(2m+1)!}{(2i)!}-(2m+1)-\displaystyle\sum_{i=m+1}^{\infty}\frac{1}{(2m+2)(2m+3)\dots(2i)}\\
&=1-\displaystyle\sum_{i=m+1}^{\infty}\frac{1}{(2m+2)(2m+3)\dots(2i)}\pmod{2}.
\end{align*}
To reach the desired conclusion, it suffices to demonstrate that \[\frac{2}{3}\geq{\displaystyle\sum_{i=m+1}^{\infty}\frac{1}{(2m+2)(2m+3)\dots(2i)}}>0.\] It is clear that the sum is positive. It can also be bounded above by the following convergent geometric series:
    \begin{align*}
        \displaystyle\sum_{i=m+1}^{\infty}\frac{1}{(2m+2)^{2i-2m-1}}&=\frac{\frac{1}{2m+2}}{1-\frac{1}{(2m+2)^2}}=\frac{2m+2}{(2m+2)^2-1}\\
        &\leq{\frac{4}{4^2-1}}=\frac{4}{15}.
    \end{align*}
    
Only $k=1$ remains to be checked. We have established that $1-\frac{1}{8}\leq{2\alpha}<1 \pmod{2}$. This means that either $\frac{7}{16}\leq{\alpha}<\frac{1}{2}\pmod{2}$ or $1+\frac{7}{16}\leq{\alpha}<1+\frac{1}{2}\pmod{2}$. If the latter statement is true, then $\alpha>1$ or $\alpha<0$ but this is false because $\alpha<2-\frac{e}{2}<2-\frac{2}{2}=1$ and $\alpha>2-\frac{3}{2}-\frac{1}{2(2)}=\frac{1}{4}$. Therefore, $\frac{7}{16}\leq{\alpha}<\frac{1}{2}\pmod{2}$, giving us $\frac{1}{3}\leq{\alpha}<1\pmod{2}$ as desired. 
\end{proof}
\begin{proof} of Proposition \ref{fact}\\
Suppose that each $n\in\mathbb{Z}^+$ has been colored according to the parity of $\lfloor{n\alpha\rfloor}$ and that there exist $a_1,a_2,a_3,a_4$ all the same color such that the consecutive gaps lie in $D=\{k!\mid k\in\mathbb{Z}^+\}$. Since this sequence is monochromatic, $\lfloor{a_{i+1}\alpha}\rfloor-\lfloor{a_i\alpha}\rfloor$ is even for $i=1,2,3$. By Lemma \ref{1/3bound}, we know that for $i=1,2,3$, \[2n_i+\frac{1}{3}\leq{(a_{i+1}-a_i)\alpha}<2n_i+1\] for some integers $n_i$.

If $a_i\alpha\in [\frac{2}{3},1)\pmod{2}$, then $a_{i+1}\alpha\in[1,2)\pmod{2}$. Similarly, if $a_i\alpha\in [\frac{5}{3},2)\pmod{2}$, then $a_{i+1}\alpha\in[0,1)\pmod{2}$. Thus, in order for $a_{i+1}$ and $a_i$ to be the same color, we need $0\leq{a_i\alpha}<\frac{2}{3}\pmod{1}$. In particular, we need $0\leq{a_2\alpha}<\frac{2}{3}\pmod{1}$ and $0\leq{a_3\alpha}<\frac{2}{3}\pmod{1}$.

By similar reasoning, these inequalities force $0\leq{a_1\alpha}<\frac{1}{3}\pmod{1}$ and $0\leq{a_2\alpha}<\frac{1}{3}\pmod{1}$, respectively. However, since $\lfloor{a_2\alpha}\rfloor$ and $\lfloor{a_1\alpha}\rfloor$ have the same parity, this gives $\frac{-1}{3}<a_2\alpha-a_1\alpha<\frac{1}{3}\pmod{2}$, contradicting the restriction that $2n_1+\frac{1}{3}\leq{(a_2-a_1)\alpha}<2n_1+1$. Therefore, no such $a_1,a_2,a_3,a_4$ exist with respect to this coloring. We have successfully avoided arbitrarily long monochromatic $D$-diffsequences, in particular by avoiding monochromatic $D$-diffsequences of length at least four.
\end{proof}

	\section{Dividing Sequences}\label{sec4}
	
	In this section, we prove Theorem \ref{1.3}, providing a complete classification of when a set $D_{\{a_n\}}:=\{\prod_{i=1}^k a_i \mid k\in \mathbb{Z}^+\}$ is $2$-accessible. First we show that $D_{\{a_n\}}$ is $2$-accessible whenever $\{a_n\}$ contains arbitrarily long strings of consecutive $2$'s. This follows from a similar argument to the one Landman and Robertson used to show that $\{2^i\mid i\in\mathbb{Z}^+\}$ is $2$-accessible \cite{LR03}.
	
	\begin{lem}
	If $\{a_n\}$ contains arbitrarily long strings of consecutive $2$'s, then $D_{\{a_n\}}$ is $2$-accessible.
	\end{lem}
	\begin{proof}
	Suppose for the sake of contradiction that some coloring of $\mathbb{Z}^+$ avoids arbitrarily long monochromatic $D_{\{a_n\}}$-diffsequences. Suppose that the longest such diffsequence is of length $k$ and given by $b_1<b_2<\dots<b_k$. 
	
	Because $\{a_n\}$ contains arbitrarily long strings of consecutive $2$'s, there exists some index $j$ such that $a_j,a_{j+1},\dots,a_{j+k-1}$ all equal $2$. Letting $C=\displaystyle\prod_{i=1}^{j-1}a_i$, we have that $2^tC\in D_{\{a_n\}}$ for $t=0,1,\dots,k$. To avoid extending our monochromatic diffsequence $b_1<b_2<\dots<b_k$ to a longer one, $b_k+2^tC$ must be the other color for $t=0,1,\dots,k$.
	However, this gives a monochromatic sequence in the other color. The consecutive differences are 
	\[
	    (b_k+2^tC)-(b_k+2^{t-1}C)=2^{t-1}C\]
	    
	    for $t=1,2,\dots,k$. As stated above, these are all elements of $D_{\{a_n\}}$ meaning we have a monochromatic $D_{\{a_n\}}$-diffsequence of length $k+1$, giving a contradiction.

	\end{proof}
	
	We are left to consider $D_{\{a_n\}}$ in the case where $\{a_n\}$ does not have arbitrarily long strings of consecutive $2$'s. All elements of $D_{\{a_n\}}$ are multiples of $a_1$, so any monochromatic $D_{\{a_n\}}$-diffsequence stays within a residue class modulo $a_1$. Considering each residue class separately, we can avoid the presence of arbitrarily long monochromatic $D_{\{a_n\}}$-diffsequences in a residue class as long as it possible to avoid arbitrarily long monochromatic $T_{\{a_n\}}$-diffsequences in $\mathbb{Z}^+$ where $T_{\{a_n\}}=\{\frac{d}{a_1}\mid d\in D_{\{a_n\}}\}$.
	Thus, it suffices to only consider when $a_1=1$.
	
	If there is no string of $k$ consecutive $2$'s in $\{a_n\}$, we will show that there exists a coloring of $\mathbb{Z}^+$ that avoids monochromatic $D_{\{a_n\}}$-diffsequences of length $2^k+1$. As in Section \ref{prev}, we will use colorings related to Beatty sequences. Every $n\in\mathbb{Z}^+$ will be colored according to the parity of $\lfloor{n\alpha\rfloor}$ where our choice of $\alpha$ depends on $\{a_n\}$. 
	
	\begin{lem}\label{1/power}
	Let $\{a_n\}$ be a sequence of positive integers with $a_1=1$ and let $d_t:=\displaystyle\prod_{i=1}^t a_i$. Let $k$ be the smallest positive integer such that $\{a_n\}$ has no string of $k$ consecutive $2$'s.
	
	Then, there exists some $\alpha>0$, depending on $\{a_n\}$, such that for all $t\in\mathbb{Z}^+$, there exists some integer $n_t$ such that
	\[
	2n_t+\frac{1}{2^k}\leq{d_t\alpha}\leq{2n_t+1}.
	\]
	\end{lem}

	\begin{proof}
	
	Our plan is to construct a series of nested intervals $J_t:=[C_t,D_t]$, contained within $[\frac{1}{2^k},1]$, such that if $\gamma\in J_t$, then $d_h\gamma\in J_t \pmod{2}$ for $h=1,2,\cdots,t$. Any number which lies in $\bigcap_{t=1}^{\infty} J_t$ can then be chosen as $\alpha$.
	
	First, for each $t\in\mathbb{Z}^+$, we will construct closed intervals $I_b^t:=[C_b^t,D_b^t]$, for $b=1,\dots,t$, such that if $d_b\gamma\in I_b^t\pmod{2}$, then $d_h\gamma\in I_h^t\pmod{2}$, for $h=b+1,b+2,\dots,t$. These will be constructed in such a way that $\frac{1}{2^k}\leq{C_b^t}<D_b^t\leq{1}$, for $b=1,\dots,t$.
	
	Now, for each $t$, we will choose $J_t$ to be $I_1^t$. Thus, for $\gamma\in J_t$, we have $d_1\gamma\in I_1^t\pmod{2}$, which means that $d_h\gamma\in I_h^t\pmod{2}\subset [\frac{1}{2^k},1]\pmod{2}$ for $h=1,2,\dots,t$. In our construction, we will have $J_{t+1}\subset J_t$ for all $t\in\mathbb{Z}^+$, so $C_t\leq{C_{t+1}}<{D_{t+1}}\leq{D_t}$ for $t\in\mathbb{Z}^+$.
	
	Consider the sequence $\{C_t\}$ for $t=1,2,\dots$. This is a non-decreasing sequence and it is bounded above by $1$. Therefore, it converges to some limit $\beta_1$. Similarly, the sequence $\{D_t\}$ for $t=1,2,\dots$ is non-increasing and bounded below by $0$ so it also converges to some limit $\beta_2\geq{\beta_1}$. For each $t\in\mathbb{Z}^+$, we have $C_t\leq{\beta_1}\leq{\beta_2}\leq{D_t}$. There exists some number $\alpha$ which satisfies $\beta_1\leq{\alpha}\leq{\beta_2}$, and this $\alpha$ lies in $J_t$ for all $t\in\mathbb{Z}^+$, meaning that $d_t\alpha\in[\frac{1}{2^k},1]\pmod{2}$ for all $t\in\mathbb{Z}^+$, as desired.
	
	It now suffices to construct each $I_b^t$ such that the necessary properties are satisfied:
	\begin{itemize}
	    \item[1)] For $b=1,2,\dots,t,$
	    \[
	    \frac{1}{2^k}\leq{C_b^t}<D_b^t\leq{1}.
	    \]
	    \item[2)] If $d_b\gamma\in I_b^t\pmod{2}$, then $d_h\gamma\in I_h^t\pmod{2}$ for $h=b+1,b+2,\dots,t$.
	    \item[3)] For all $t\in\mathbb{Z}^+$,
	    \[
	    J_{t+1}\subset J_t.
	    \]
	\end{itemize}

	We construct the family of $I_b^t$'s in the following way.  For each $t$, we begin with $I_t^t:=[\frac{1}{2^{k-f(t)}},1]$, where $f(t)$ denotes the smallest nonnegative integer $p$ such that $a_{t-p}\ne{2}$. Note that since $\{a_n\}$ has no string of $k$ consecutive $2$'s, we will always have that $f(t)<k$.  Then for $b\geq{2}$, we recursively construct $I_{b-1}^t:=[C_{b-1}^t,D_{b-1}^t]$ from $I_b^t:=[C_b^t,D_b^t]$ as follows:
	
	\begin{itemize}
	\item If $a_b$ is odd:
	\[
	C_{b-1}^t=\frac{a_b-1+C_b^t}{a_b}\hspace{1in}D_{b-1}^t=\frac{a_b-1+D_b^t}{a_b}.
	\]
	
	\item If $a_b$ is even:
	\[
	C_{b-1}^t=\frac{a_b-2+C_b^t}{a_b}\hspace{1in}D_{b-1}^t=\frac{a_b-2+D_b^t}{a_b}.
	\]
	\end{itemize}
	
	We now verify this construction satisfies the necessary properties:
	
	\begin{itemize}
	\item[1)]
	For $b=1,2,\dots,t$,
	\[
	\frac{1}{2^k}\leq{C_b^t}<D_b^t\leq{1}.
	\]
	
	\begin{proof}\renewcommand{\qedsymbol}{}
	We have set $C_t^t:=\frac{1}{2^{k-f(t)}}\geq{\frac{1}{2^k}}$ and $D_t^t:=1$. We have $D_b^t>C_b^t$ for $b=t$ and whenever $D_b^t>C_b^t$, we get that
	\[a_b-2+D_b^t>a_b-2+C_b^t,\hspace{1in} a_b-1+D_b^t>a_b-1+C_b^t,
	\]
	so by induction, $D_b^t>C_b^t$ for $b=1,\dots,t$.
	
	Now we need to show that $C_b^t\geq{\frac{1}{2^k}}$ and $D_b^t\leq{1}$ for $b=1,\dots,t-1$. We will proceed by induction by showing that for $b=2,\dots,t$, $C_h^t\geq{\frac{1}{2^k}}$ for $h=b,b+1,\dots,t$  implies $C_{b-1}^t\geq{\frac{1}{2^k}}$ and that $D_b^t\leq{1}$ implies $D_{b-1}^t\leq{1}$. There are two cases to consider separately:
	
	\begin{itemize}
	    \item[i)] $a_b>2$, or
	    \item[ii)] $a_b=2$. 
	\end{itemize}
	
	In either case,
	\[
	   D_{b-1}^t\leq{\frac{a_b-1+D_b^t}{a_b}}\leq{\frac{a_b-1+1}{a_b}}=1.
	\]
	
	Since $C_b^t<D_b^t\leq{1}$, we know that $1\geq{1-C_b^t}>0$. In case i) with odd $a_b$, \[C_{b-1}^t=1-\frac{1-C_b^t}{a_b}.\] This is minimized when $a_b=3$, in which case, \[C_{b-1}^t=1-\frac{1-C_b^t}{3}\geq{\frac{2}{3}}.\]
	For case i) with even $a_b$,
	\[
	C_{b-1}^t=1-\frac{2-C_b^t}{a_b},
	\]
	which is minimized when $a_b=4$, yielding,
	\[
	C_{b-1}^t=1-\frac{2-C_b^t}{4}\geq{\frac{1}{2}}.
	\]
	
	This means that when $a_b>2$, we not only get that $C_{b-1}^t\geq{\frac{1}{2^k}}$, but we get the potentially stronger statement of $C_{b-1}^t\geq{\frac{1}{2}}$. Note that this stronger fact will be useful for case ii) as well as in the proof of Property 3).
	
	Now consider case ii). We consider the longest string of consecutive $2$'s starting at $a_b$, but not going past $a_t$. There is a unique positive $q$ satisfying $b+q-1\leq{t}$ such that $a_b=a_{b+1}=a_{b+2}=\dots=a_{b+q-1}=2$, but either $a_{b+q}>2$ or $b+q=t+1$. Note that the sequence $\{a_n\}$ has no string of $k$ consecutive $2$'s so $q<k$. First, we consider when $b+q\leq{t}$, so $a_{b+q}>2$. As a consequence of case i), we have that $C_{b+q-1}^t\geq{\frac{1}{2}}$.
	
	Now since $a_b=a_{b+1}=\dots=a_{b+q-1}=2$, we have for $h=b-1,b,\dots,b+q-2$,
	\begin{align*}
	C_h^t&=\frac{a_{h+1}-2+C_{h+1}^t}{a_{h+1}}\\
	&=\frac{2-2+C_{h+1}^t}{2}=\frac{C_{h+1}^t}{2}.
	\end{align*}
	
	This means
	\begin{align*}C_{b-1}^t=\frac{C_{b+q-1}^t}{2^q}\geq{\frac{1/2}{2^q}}=\frac{1}{2^{q+1}}.
	\end{align*}
	
	Recall that $q<k$, so this quantity is at least $\frac{1}{2^k}$.
	
	Next, consider when $b+q=t+1$. As in the previous subcase, we have that $C_{b-1}^t=\frac{C_{b+q-1}^t}{2^q}$. Our initial definition of $C_t^t$ gives $C_{b+q-1}^t=\frac{1}{2^{k-f(t)}}$.
	We have $a_b=a_{b+1}=\dots=a_t=2$, but we know that $a_{t-f(t)}\ne{2}$. Thus, $b>t-f(t)$, and $q=t-b+1<f(t)+1$. Therefore,
	\[
	C_{b-1}^t=\frac{C_{b+q-1}^t}{2^q}=\frac{1/{2^{k-f(t)}}}{2^q}=\frac{1}{2^{k+(q-f(t))}}\geq{\frac{1}{2^k}},
	\]
	where the last inequality follows from the fact that $q$ and $f(t)$ are both integers.
	\end{proof}
	\item[2)]
	If $d_b\gamma\in I_b^t\pmod{2}$, then $d_h\gamma\in I_h^t\pmod{2}$ for $h=b+1,b+2,\dots,t$.
	\begin{proof}\renewcommand{\qedsymbol}{}
	It suffices to show that if $d_b\gamma\in I_b^t\pmod{2}$, then $d_{b+1}\gamma\in I_{b+1}^t\pmod{2}$.
	
	If $d_b\gamma\in I_b^t\pmod{2}$ and $a_{b+1}$ is odd,
	\begin{align*}
	  C_b^t&\leq{d_b\gamma}\leq{D_b^t}\pmod{2}\\
	  \frac{a_{b+1}-1+C_{b+1}^t}{a_{b+1}}&\leq{d_b\gamma}\leq\frac{a_{b+1}-1+D_{b+1}^t}{a_{b+1}}\pmod{2}\\
	  a_{b+1}-1+C_{b+1}^t&\leq{a_{b+1}d_b\gamma}\leq{a_{b+1}-1+D_{b+1}^t}\pmod{2}\\
	  C_{b+1}^t&\leq{d_{b+1}\gamma}\leq{D_{b+1}^t}\pmod{2}.
	\end{align*}
	
	Going from the second line to the third line uses that $0\leq{C_{b+1}^t}\leq{D_{b+1}^t}\leq{1}$. In particular, we have some integer $n_b$ where $d_b\gamma\in[2n_b+\frac{a_{b+1}-1+C_{b+1}^t}{a_{b+1}},2n_b+\frac{a_{b+1}-1+D_{b+1}^t}{a_{b+1}}]$, so multiplying by $a_{b+1}$ gives $a_{b+1}d_b\gamma\in[2n_ba_{b+1}+(a_{b+1}-1)+C_{b+1}^t,2n_ba_{b+1}+(a_{b+1}-1)+D_{b+1}^t]$.
	
	Similarly, if $d_b\gamma\in I_b^t\pmod{2}$ and $a_{b+1}$ is even,
	\begin{align*}
	  C_b^t&\leq{d_b\gamma}\leq{D_b^t}\pmod{2}\\
	  \frac{a_{b+1}-2+C_{b+1}^t}{a_{b+1}}&\leq{d_b\gamma}\leq\frac{a_{b+1}-2+D_{b+1}^t}{a_{b+1}}\pmod{2}\\
	  a_{b+1}-2+C_{b+1}^t&\leq{a_{b+1}d_b\gamma}\leq{a_{b+1}-2+D_{b+1}^t}\pmod{2}\\
	  C_{b+1}^t&\leq{d_{b+1}\gamma}\leq{D_{b+1}^t}\pmod{2},
	\end{align*}
	
	where going from the second line to the third line again uses that $0\leq{C_{b+1}^t}\leq{D_{b+1}^t}\leq{1}$.
	\end{proof}
	\item[3)] For all $t\in\mathbb{Z}^+$,
	    \[
	    J_{t+1}\subset J_t.
	    \]
	    \begin{proof}\renewcommand{\qedsymbol}{}
	    Recall that $J_t$ is $I_1^t$. We will prove the stronger claim that $I_b^{t+1}\subset I_b^t$ for all $b=1,2,\dots,t$. We will proceed by induction by showing the base case $I_t^{t+1}\subset I_t^t$ and that $I_b^{t+1}\subset I_b^t$ implies $I_{b-1}^{t+1}\subset I_{b-1}^t$ for $b=2,3,\dots,t$.
	    
	    For the base case, we must show that $C_t^{t+1}\geq{C_t^t}$ and $D_t^{t+1}\leq{D_t^t}$. By construction, $D_t^t=1$ and we have already shown property 1) which gives that $D_t^{t+1}\leq{1}$. Now we will show that $C_t^{t+1}\geq{C_t^t}$.
	    
	    By construction, $C_t^t=\frac{1}{2^{k-f(t)}}\leq{\frac{1}{2}}$. When $a_{t+1}>2$, case i) of the proof of property 1) gives that $C_t^{t+1}\geq{\frac{1}{2}}\geq{C_t^t}$. When $a_{t+1}=2$, we have \[C_t^{t+1}=\frac{C_{t+1}^{t+1}}{2}=\frac{1/2^{k-f(t+1)}}{2}=\frac{1}{2^{k-f(t+1)+1}}.\]
	    
	    By the definition of $f(t)$, we have that $a_{t-f(t)}=a_{t+1-(f(t)+1)}\ne 2$, but the last $f(t)$ elements in the finite sequence $a_1,a_2,\dots,a_t$ are $2$'s. $a_{t+1}=2$ as well so the last $f(t)+1$ elements in the finite sequence $a_1,a_2,\dots,a_{t+1}$ are $2$'s. This yields $f(t+1)=f(t)+1$. Therefore, $C_t^{t+1}=\frac{1}{2^{k-f(t+1)+1}}=\frac{1}{2^{k-(f(t)+1)+1}}=\frac{1}{2^{k-f(t)}}={C_t^t}$, and the base case is complete.
	    
	    If we have $I_b^{t+1}\subset I_b^t$, then $C_{b}^{t+1}\geq{C_b^t}$ and $D_{b}^{t+1}\leq{D_b^t}$. Depending on the parity of $a_b$, we have either $C_{b-1}^{t+1}=\frac{a_b-1+C_b^{t+1}}{a_b}$ and $C_{b-1}^t=\frac{a_b-1+C_b^t}{a_b}$, or $C_{b-1}^{t+1}=\frac{a_b-2+C_b^{t+1}}{a_b}$ and $C_{b-1}^t=\frac{a_b-2+C_b^t}{a_b}$. In either case, the expression for $C_{b-1}^{t+1}$ has at least as large a numerator with the same denominator, so $C_{b-1}^{t+1}\geq{C_{b-1}^t}$.
	    
	    Similarly, we have either $D_{b-1}^{t+1}=\frac{a_b-1+D_b^{t+1}}{a_b}$ and $D_{b-1}^t=\frac{a_b-1+D_b^t}{a_b}$, or $D_{b-1}^{t+1}=\frac{a_b-2+D_b^{t+1}}{a_b}$ and $D_{b-1}^t=\frac{a_b-2+D_b^t}{a_b}$. In either case, the expression for $D_{b-1}^t$ has at least as large a numerator with the same denominator, so $D_{b-1}^{t}\geq{D_{b-1}^{t+1}}$.
	    Thus, $I_b^{t+1}\subset I_b^t$ implies $I_{b-1}^{t+1}\subset I_{b-1}^t$, and the induction is complete.
	    \end{proof}
	\end{itemize}

 	\end{proof}
 	
 	We are now ready to complete the proof of Theorem \ref{1.3}. In particular, we show that when $\{a_n\}$ lacks arbitrarily long strings of consecutive $2$'s, then there exists a coloring of $\mathbb{Z}^+$ which avoids arbitrarily long monochromatic $D_{\{a_n\}}$-diffsequences.
	
	\begin{proof}For ${\{a_n\}}$ with no string of $k$ consecutive $2$'s, choose some $\alpha$ satisfying the statement of Lemma \ref{1/power} such that $d\alpha\in[\frac{1}{2^k},1]\pmod{2}$ for $d\in D_{\{a_n\}}$. Color each $n\in\mathbb{Z}^+$ according to the parity of $\floor{n\alpha}$. Suppose for the sake of contradiction that we have a monochromatic $D_{\{a_n\}}$-diffsequence $b_1<b_2<\dots<b_{2^k+1}$ with respect to this coloring. Since this sequence is monochromatic, $\lfloor{b_{i+1}}\alpha\rfloor-\lfloor{b_i\alpha\rfloor}$ is even for $i=1,2,\dots, 2^k$. By Lemma \ref{1/power}, we know that for $i=1,2,\dots,2^k$,
	\begin{equation}\label{eq}
	2m_i+\frac{1}{2^k}\leq{(b_{i+1}-b_i)\alpha}\leq{2m_i+1}
	\end{equation}
	for some integers $m_i$. In order for $b_{i+1}$ and $b_i$ to be the same color, we need $0\leq{b_i\alpha}<1-\frac{1}{2^k}\pmod{1}$, so $0\leq{b_{2^k}\alpha}<1-\frac{1}{2^k}\pmod{1}$. Since $\lfloor{b_{2^k}\alpha}\rfloor$ and $\lfloor{b_{2^k-1}\alpha}\rfloor$ have the same parity, this in turn necessitates
	
	\[
	0\leq{(b_{2^k-1})\alpha}<1-\frac{2}{2^k}\pmod{1},
	\]
	
	which necessitates
	
		\[
	0\leq{(b_{2^k-2})\alpha}<1-\frac{3}{2^k}\pmod{1}.
	\]
	
	Continuing in this manner, we get that
	
	\[
	0\leq{b_2\alpha}<1-\frac{2^k-1}{2^k}=\frac{1}{2^k}\pmod{1}.
	\]
	
	However, $\lfloor{b_1\alpha}\rfloor$ and $\lfloor{b_2\alpha}\rfloor$ have the same parity so either $0\leq{b_2\alpha-b_1\alpha}<\frac{1}{2^k}\pmod{2}$ or $1<b_2\alpha-b_1\alpha<2\pmod{2}$. This directly contradicts Equation \ref{eq}.
	\end{proof}

	\section{Concluding Remarks}\label{cr}
	In this paper, we improved the lower bound for $\Delta(D,k)$ when $D=\{2^i\mid i\in\mathbb{Z}_{\geq{0}}\}$, establishing that $\Delta(D,k)$ grows faster than any polynomial. The exponent of the lower bound is $\Theta(\sqrt{k})$ while the exponent of the upper bound is $\Theta(k)$. Values computed up to $k=12$ align closely with the lower bound so we conjecture that it is asymptotically tight:
	\begin{conj}
	When $D=\{2^i\mid i\in\mathbb{Z}_{\geq{0}}\}$,
	\[
	\Delta(D,k)=2^{\Theta(\sqrt{k})}.
	\]
	\end{conj}
	
	It does not make sense to consider $D=\{n^i\mid i\in\mathbb{Z}_{\geq{0}}\}$ for any other integer $n$ because $n^i\equiv 1\pmod{n-1}$, allowing us to use a periodic coloring modulo $n-1$ to show $D$ is not $2$-accessible. However, it may be interesting to consider other sets that exhibit exponential growth such as $D_{\alpha}^f:=\{\lfloor{\alpha^i}\rfloor\mid i\in\mathbb{Z}_{\geq{0}}\}$ or $D_{\alpha}^c=\{\lceil{\alpha^i}\rceil\mid i\in\mathbb{Z}_{\geq{0}}\}$ for $\alpha>1$.
	
	When $\alpha=2^{\frac{1}{n}}$ for some $n\in{\mathbb{Z}_{\geq{0}}}$, then both $D_{\alpha}^f$ and $D_{\alpha}^c$ contain all powers of $2$ and therefore are $2$-accessible. More generally, we may consider $D_{\delta, \alpha}^f:=\{\lfloor{\delta\alpha^i}\rfloor\mid i\in\mathbb{Z}_{\geq{0}}\}\cap\mathbb{Z}_{>0}$ and $D_{\delta,\alpha}^c=\{\lceil{\delta\alpha^i}\rceil\mid i\in\mathbb{Z}_{\geq{0}}\}$ for $\alpha>1, \delta>0$. While not quite in this form, the set of Fibonacci numbers, which is $2$-accessible \cite{LR03}, is similar to $D_{\delta,\alpha}^f$ and $D_{\delta,\alpha}^c$ for $\alpha=\frac{1+\sqrt{5}}{2}$ and $\delta=\frac{1}{\sqrt{5}}$.
	
	\begin{ques}
	Is there any $\alpha\in (1,2)$ for which $D_{\delta,\alpha}^f$ or $D_{\delta,\alpha}^c$ is not $2$-accessible for some choice of $\delta>0$? (for $\delta=1$?)
	\end{ques}
	
	\begin{ques}
	Is there any $\alpha>2$ for which $D_{\delta,\alpha}^f$ or $D_{\delta,\alpha}^c$ is $2$-accessible for some choice of $\delta>0$? (for $\delta$=1?)
	\end{ques}
	
	If the answer to either of these questions is yes, we can also consider whether the set of such $\alpha$ has nonzero measure.
	
	While we have, in Theorem \ref{1.3}, demonstrated a wide class of sets that are not $2$-accessible, this occurs for divisibility reasons, as opposed to sparsity reasons. It would be interesting to determine whether there exists a general growth condition on the elements of a set $D$ which can disqualify it from being $2$-accessible. For example, we propose the following question.
 	
 	\begin{ques}
 	Does there exist an absolute constant $C\geq{2}$ such that there is no $2$-accessible set $D=\{d_1,d_2,d_3,\cdots\}$ satisfying $d_{i+1}>Cd_i$ for all $i\in\mathbb{Z}^+$? If yes, what is the smallest such $C$?
 	\end{ques}

	For sets $D$ that are $r$-accessible, there are examples where $\Delta(D,k;r)$ is polynomial in $k$ such as $D=\{a\in\mathbb{Z}^+\mid m\nmid a\}$ for any integer $m\geq{3}$ and $r=2$ \cite{CCLS18}, $D=\{2\}\cup\{2a-1\mid a\in\mathbb{Z}^+\}$ for $r\leq{3}$ \cite{Lan97}, and $D=\{ma\mid a\in\mathbb{Z}^+\}$ for any $m\in\mathbb{Z}^+$ and any $r\in\mathbb{Z}^+$ . There are also cases like $D=\{2^i\mid i\in\mathbb{Z}_{\geq{0}}\}$ with $r=2$ where $\Delta(D,k;r)$ is at most exponential in $k$, but still grows faster than any polynomial. We consider what other types of behavior are possible for the function $\Delta(D,k;r)$ for sets $D$ that are $r$-accessible.
	
	\begin{ques}
	For which $t\in\mathbb{R}^+$ does there exist an $r$-accessible set $D\subset\mathbb{Z}^+$ for which
	\[\Delta(D,k;r)=e^{\Omega(k^t)}\hspace{0.1in}?\]
	\end{ques}
	
	
	One possible approach is as follows. For any infinite set $T\in\mathbb{Z}^+$, the set $D=\{i-j\mid i,j\in T, i>j\}$ is $r$-accessible for any $r$ \cite{LR03}. 
	This means we can construct arbitrarily sparse sets that are still $r$-accessible and we generally expect a sparser set to have higher values for $\Delta(D,k;r)$.

	\textbf{Acknowledgments}. The author would like to thank Bruce Landman for introducing him to diffsequences at the 2015 University of West Georgia REU.
	 \bibliographystyle{alpha}
	 \bibliography{references}

\end{document}